\documentclass[11pt,righttag]{amsart}
\usepackage{diagrams}
\usepackage{graphicx}
\usepackage{stmaryrd}

\begin{document}

\topmargin-0.1in
\textheight8.5in
\textwidth5.5in
\footskip35pt
\oddsidemargin.5in
\evensidemargin.5in

\newcommand{\V}{{\cal V}}      % cal in math mode
\renewcommand{\O}{{\cal O}}
\newcommand{\LL}{\cal L}
\newcommand{\Ext}{\hbox{Ext}}
\newcommand{\Hom}{\hbox{Hom}}
\newcommand{\Proj}{\hbox{Proj}}
\newcommand{\GrMod}{\hbox{GrMod}}
\newcommand{\grmod}{\hbox{gr-mod}}
\newcommand{\tors}{\hbox{tors}}
\newcommand{\rank}{\hbox{rank}}
\newcommand{\End}{\hbox{End}}
\newcommand{\GKdim}{\hbox{GKdim}}
\newcommand{\im}{\hbox{im}}
\renewcommand{\ker}{\hbox{ker}}

\newcommand{\lonto}{{\protect \longrightarrow\!\!\!\!\!\!\!\!\longrightarrow}}

\newcommand{\m}{{\mu}}
\newcommand{\gl}{{\frak g}{\frak l}}
\newcommand{\ssl}{{\frak s}{\frak l}}

\newcommand{\ds}{\displaystyle}
\newcommand{\s}{\sigma}
\renewcommand{\l}{\lambda}
\renewcommand{\a}{\alpha}
\renewcommand{\b}{\beta}
\newcommand{\G}{\Gamma}
\newcommand{\g}{\gamma}
\newcommand{\z}{\zeta}
\newcommand{\e}{\epsilon}
\renewcommand{\d}{\delta}
\newcommand{\p}{\rho}
\renewcommand{\t}{\tau}

\newcommand{\C}{{\mathbb C}}
\newcommand{\N}{{\mathbb N}}
\newcommand{\Z}{{\mathbb Z}}
\newcommand{\ZZ}{{\mathbb Z}}
\newcommand{\K}{{\mathcal K}}

\newcommand{\rowxy}{(x\ y)}
\newcommand{\colxy}{ \left({\begin{array}{c} x \\ y \end{array}}\right)}
\newcommand{\scolxy}{\left({\begin{smallmatrix} x \\ y
\end{smallmatrix}}\right)}

\renewcommand{\P}{{\Bbb P}}

\newcommand{\la}{\langle}
\newcommand{\ra}{\rangle}
\newcommand{\tensor}{\otimes}

\newarrow{Equals}{=}{=}{=}{=}{=}

\newtheorem{thm}{Theorem}[section]
\newtheorem{lemma}[thm]{Lemma}
\newtheorem{cor}[thm]{Corollary}
\newtheorem{prop}[thm]{Proposition}

\theoremstyle{definition}
\newtheorem{defn}[thm]{Definition}
\newtheorem{notn}[thm]{Notation}
\newtheorem{ex}[thm]{Example}
\newtheorem{rmk}[thm]{Remark}
\newtheorem{rmks}[thm]{Remarks}
\newtheorem{note}[thm]{Note}
\newtheorem{example}[thm]{Example}
\newtheorem{problem}[thm]{Problem}
\newtheorem{ques}[thm]{Question}
\newtheorem{thingy}[thm]{}

\newcommand{\onto}{{\protect \rightarrow\!\!\!\!\!\rightarrow}}
\newcommand{\donto}{\put(0,-2){$|$}\put(-1.3,-12){$\downarrow$}{\put(-1.3,-14.5) 

{$\downarrow$}}}

\newcounter{letter}
\renewcommand{\theletter}{\rom{(}\alph{letter}\rom{)}}

\newenvironment{lcase}{\begin{list}{~~~~\theletter} {\usecounter{letter}
\setlength{\labelwidth4ex}{\leftmargin6ex}}}{\end{list}}

\newcounter{rnum}
\renewcommand{\thernum}{\rom{(}\roman{rnum}\rom{)}}

\newenvironment{lnum}{\begin{list}{~~~~\thernum}{\usecounter{rnum}
\setlength{\labelwidth4ex}{\leftmargin6ex}}}{\end{list}}

%\begin{document}

%\large{Version as of April 27, 2010} 

\title{$A_{\infty}$-algebra structures associated to $\K_2$ algebras }


\author[  Conner, Goetz ]{ }
\maketitle

\begin{center}

\vskip-.2in Andrew Conner$^\dagger$ and
 Pete Goetz$^\ddagger$ \\
\bigskip
$^\dagger$ Department of Mathematics \\ University of Oregon \\ Eugene, Oregon 97403
\bigskip

$^\ddagger$ Department of Mathematics\\ Humboldt State University\\
Arcata, California  95521
\\ \ \\

\end{center}

\setcounter{page}{1}

\thispagestyle{empty}

%\vspace{0.2in}

\begin{abstract} The notion of a $\K_2$-algebra was recently introduced by Cassidy and Shelton as a generalization of the notion of a Koszul algebra. The Yoneda algebra of any connected graded algebra admits a canonical $A_{\infty}$-algebra structure. This structure is trivial if the algebra is Koszul. We study the $A_{\infty}$-structure on the Yoneda algebra of a $\K_2$-algebra. For each non-negative integer $n$ we prove the existence of a $\K_2$-algebra $B$ and a canonical $A_{\infty}$-algebra structure on the Yoneda algebra of $B$ such that the higher multiplications $m_i$ are nonzero for all $3 \leq i \leq n+3$. We also provide examples which show that the $\K_2$ property is not detected by any obvious vanishing patterns among higher multiplications.
\baselineskip15pt

\end{abstract}

\bigskip

\baselineskip18pt

%%%%%%%%%%%%%%%%%%%%%%%%%%%%%%%%%%%%%%%%%%%%%%%%%%%%%%%%%%%%%%%%%%%%%%
\section{Introduction}

Let $\N$ denote the set of non-negative integers. Let $K$ be a field. In this paper the term {\it $K$-algebra} will mean a connected, $\N$-graded, $K$-algebra. Let $A$ be a $K$-algebra. The {\it augmentation ideal} is $A_+ = \oplus_{i \ge 1} A_i$.  The term {\it $A$-module} will mean a  graded left $A$-module. The {\it trivial module} is $_AK = A/A_+$. For $A$-modules $M$ and $N$, and for $n\in\Z$, an $A$-linear map $f:M\rightarrow N$ such that $f(M_i)\subset N_{i+n}$ for all $i\in\N$ will be called a {\it degree $n$ homomorphism\/}.  The {\it Yoneda algebra of}   $A$ is the bigraded cohomology algebra $E(A) = \oplus E^{p,q}(A)=\oplus \Ext^{p, q}_A(K,K)$ (where $p$ is the cohomology degree and $-q$ is the internal degree inherited from the grading on $A$). Let $E^p(A) = \oplus_q E^{p,q}(A)$.

The notion of a $\K_2$-algebra was recently introduced by Cassidy and Shelton in \cite{CS} as a generalization of Koszul algebras. 

\begin{defn} Let $A$ be a $K$-algebra which is finitely generated in degree 1. We say $A$ is a $\K_2$-algebra if $E(A)$ is generated as an algebra by $E^1(A)$ and $E^2(A)$. 
\end{defn}

Let $A$ be a $K$-algebra. We recall that $A$ is Koszul if and only if $E(A)$ is generated as an algebra by $E^1(A)$. This forces a Koszul algebra to be a quadratic algebra. Berger \cite{B} studied a generalization of  the Koszul property for nonquadratic algebras called $N$-Koszul algebras ($N > 2$). In \cite{G}, Green, Marcos, Mart\'{i}nez-Villa, and Zhang prove: An algebra $A$ is $N$-Koszul if and only if:

\begin{enumerate}
\item $A$ has defining relations all of degree $N$, and
\item $E(A)$ is generated as an algebra by $E^1(A)$ and $E^2(A)$.
\end{enumerate}
Hence the class of $\K_2$-algebras contains the class of Koszul algebras and the class of $N$-Koszul algebras but allows for algebras with relations of different degrees. 

Cassidy and Shelton \cite{CS} prove the class of $\K_2$-algebras is closed under many standard operations in ring theory: tensor products, regular normal extensions, and graded Ore extensions.
The class of $\K_2$-algebras includes Artin-Schelter regular algebras of global dimension 4 on three linear generators as well as graded complete intersections. Cassidy and Shelton characterize $\K_2$-algebras in terms of a minimal projective resolution of the trivial module. 
We state their criterion in Section 3.

Another important recent development in ring theory is the use of $A_{\infty}$-algebras (cf. \cite{HL}, \cite{LPWZ2}, \cite{LPWZ1}). The notion of an $A_\infty$-algebra was first defined by Stasheff in \cite{St}. 

\begin{defn}  An $A_\infty$-algebra over a field $K$ is a 
 $\Z$-graded vector space $E = \oplus_{p \in \Z} E^p$
 together with graded $K$-linear maps
$$m_n: E^{\tensor n} \to E, \ \ n \geq 1$$ 
of degree $2-n$  which satisfy the Stasheff identities
$$\mathbf{SI}(n)\qquad \sum (-1)^{r+st} m_u(1^{\tensor r} \tensor m_s \tensor 1^{\tensor t}) = 0.$$
where the sum runs over all decompositions $n = r+s+t$ such that $r, t \geq 0$, $s \geq 1$, and $u = r + 1 + t$. 

\end{defn}

The linear maps $m_n$ for $n\ge 3$ are called \emph{higher multiplications.}
The $A_{\infty}$-structure $\{m_i\}$ on a graded $K$-vector space $E$ determines a $K$-linear map $\oplus m_i: T(E)_+\rightarrow E$. This map induces the structure of a differential coalgebra on $T(E)$ (see \cite{K}).

We note that the identity {\bf SI$(1)$} is $m_1 m_1 = 0$ and the degree of $m_1$ is 1 so $m_1$ is a differential on $E$. The identity {\bf SI$(2)$} can be written as
$$m_1m_2 = m_2(m_1 \tensor 1 + 1 \tensor m_1).$$ 
Thus the differential $m_1$ is a graded derivation with respect to $m_2$. Hence a differential graded algebra is an $A_{\infty}$-algebra with $m_n = 0$ for all $n \geq 3$. The map $m_2$ plays the role of multiplication in an $A_{\infty}$-algebra, but in general $m_2$ is not associative. However, it is clear from {\bf SI$(3)$} that $m_2$ is associative if $m_1$ or $m_3$ is zero.

An important result in the theory of $A_{\infty}$-algebras is the following theorem due to Kadeishvili.

\begin{thm}[\cite{K}] 
\label{Kadeishvili}
Let $(E, \{m_i^E\})$ be an $A_{\infty}$-algebra. The cohomology $H^* E$ with respect to $m^E_1$ admits an $A_{\infty}$-algebra structure $\{m_i\}$ such that
\begin{enumerate}
\item $m_1 = 0$ and  $m_2$ is induced by $m_2^E$, and
\item there is a quasi-isomorphism of $A_{\infty}$-algebras $H^* E \to E$ lifting the identity of $H^* E$.
\end{enumerate}
Moreover, this structure is unique up to (non unique) isomorphism of $A_{\infty}$-algebras.
\end{thm}

We refer the interested reader to \cite{Keller2} for definitions of  \emph{morphism}, \emph{quasi-isomorphism}, and \emph{isomorphism} of $A_{\infty}$-algebras. We will not use these notions in this paper. We call an $A_{\infty}$-algebra structure on $H^*E$ \emph{canonical} if it belongs to the $A_{\infty}$-isomorphism class of structures provided by the theorem.

Let $A$ be a $K$-algebra. We recall that the Yoneda algebra of $A$, $E(A)$, is the cohomology algebra of a differential graded algebra (see Section 4). Therefore Kadeishvili's theorem implies that $E(A)$ admits a canonical $A_{\infty}$-algebra structure. We remark that $E(A)$ is bigraded. The grading on the maps $m_n$ refers to the cohomological degree. By the construction described in Section 4, we may assume the internal degree of $m_n$ is 0 for all $n$. 

The following theorem was the main motivation for this paper.

\begin{thm} Let $A$ be a $K$-algebra which is finitely generated in degree 1. Let $E(A)$  be the Yoneda algebra of $A$. Then $A$ is a Koszul algebra if and only if every canonical $A_{\infty}$-algebra structure on $E(A)$  has $m_n = 0$ for all $n \ne 2$.
\end{thm}

One direction of this theorem is clear. Suppose that $A$ is a Koszul algebra. Then the Yoneda algebra is concentrated on the diagonal $\oplus E^{p,p}(A)$. It follows, purely for degree reasons, that all multiplications other than $m_2$ are zero.

The converse is implied by work of May and Gugenheim \cite{GugMay1} and Stasheff \cite{Stasheff1}. It can also be proved using the fact that $E(A)$ is $A_{\infty}$-generated by $E^1(A)$. This fact appears in Keller's paper \cite{Keller1} without proof. A constructive proof that $E(A)$ is $A_{\infty}$-generated by $E^1(A)$ will appear in the first author's Ph.D. thesis using the ideas contained in Section 4.

The goal of this paper is to provide some partial answers to the following natural questions.

\begin{itemize}
\item What restrictions does the $\K_2$ condition place on a canonical $A_{\infty}$-structure on the Yoneda algebra?
\item Do certain $A_{\infty}$-structures on the Yoneda algebra guarantee the original algebra is $\K_2$?
\end{itemize}
The analogues of these questions for $N$-Koszul algebras with $N\ge 3$ were considered by He and Lu in \cite{HL}, where they obtain a result similar to Theorem 1.4.  They prove an algebra $A$ is $N$-Koszul if and only if $E(A)$ is a reduced $(2, N)$ $A_{\infty}$-algebra and is $A_{\infty}$- generated by $E^1$. In a $(2, N)$ $A_{\infty}$-algebra, all multiplications other than $m_2$ and $m_N$ are zero. Their result is aided by the fact that for an $N$-Koszul algebra, $E^p(A)$ is concentrated in a single internal degree. The Yoneda algebra of a $\K_2$-algebra does not generally satisfy such a strong purity condition. 

Our main results are the following.

\begin{thm} 
\label{MainThm1}
For each $n \in \N$ there exists a $\K_2$-algebra $B$ such that 
\begin{enumerate}
\item The defining relations of $B$ are quadratic and cubic, and 
\item $E(B)$ has a canonical $A_{\infty}$-algebra structure such that $m_{i}$ is nonzero for all $3 \leq i \leq n+3$.
\end{enumerate}
\end{thm}

This shows, in contrast to the cases of Koszul and $N$-Koszul algebras, that vanishing of higher multiplications on the Yoneda algebra of a $\K_2$-algebra need not be determined in any obvious way by the degrees of defining relations. Recently, Green and Marcos \cite{Green} defined the notion of $2$-$d$-\emph{Koszul} algebras. We remark that the algebra $B$ of Theorem \ref{MainThm1} is a $2$-$d$-Koszul algebra.

\begin{thm}
\label{MainThm2}
There exist $K$-algebras $A^1$ and $A^2$  such that
\begin{enumerate}
\item $A^1$ is $\K_2$ and $A^2$ is not $\K_2$,
\item $A^1$ and $A^2$ have the same Poincar\'{e} series, and
\item $E(A^1)$ and $E(A^2)$ admit canonical $A_{\infty}$-structures which are not distinguished by $m_n$ for $n\ge 3$.
\end{enumerate}
\end{thm}

These examples demonstrate that obvious patterns of vanishing among higher multiplications cannot detect the $\K_2$ property. 

We are now ready to outline the paper.
In Section 2 we introduce the $K$-algebra $B$ and prove some basic facts about a canonical graded $K$-basis for $B$. Section 3 is the technical heart of the paper. We describe a minimal projective resolution of the trivial $B$-module and prove that $B$ is a $\K_2$-algebra. In Section 4 we prove some general results on using a minimal resolution of the trivial module to compute $A_{\infty}$-algebra structures on Yoneda algebras. In Section 5 we compute part of an $A_{\infty}$-algebra structure on $E(B)$ and finish the proof of Theorem \ref{MainThm1}. 
Finally in Section 6 we prove Theorem \ref{MainThm2}.

We thank Brad Shelton for several helpful conversations and for reading a preliminary draft of this paper. 

\section{The Algebra $B$}

We begin this section by defining an algebra $B$. We exhibit a canonical $K$-basis for $B$ and use it to compute left annihilator ideals of certain elements of $B$.

Let $K$ be a field and fix $n \in \N$. Let $K^*$ denote the set of nonzero elements of $K$. Let $V$ be a $K$-vector space with basis $X=\{ a_i, b_i, c_i\}_{0\le i\le n}$. It will be convenient to define elements $b_s$ of $V$ to be 0 for $s>n$. Likewise we define $c_t$ to be $0$ in $V$ for $t<0$. Let $T(V)$ be the tensor algebra on $V$, graded by tensor degree. Let $R\subset V^{\otimes 2}\oplus V^{\otimes 3}$ be the set of tensors 
$$
\{a_nb_nc_n, c_0a_0\}\cup 
\{a_ib_ic_i + c_{i+1}a_{i+1}b_{i+1},  b_{i+1}c_{i+1}a_{i+1}, c_ic_{i+1}, b_{i+1}a_i \}_{0 \leq i < n}.
$$ 
We define $B$ to be the $K$-algebra $T(V)/I$ where $I$ is the ideal of $T(V)$ generated by $R$. The ideal $I$ is homogeneous, so $B$ inherits a grading from $T(V)$. We have the canonical quotient map $\pi_B: T(V)\rightarrow B$ of graded $K$-algebras. 

The basis $X$ generates a free semigroup $\la X\ra$ with 1 under the multiplication in $T(V)$. We call elements of $\la X\ra$ \emph{pure tensors}. Let $\la X \ra_d$ denote the set of pure tensors of degree $d$.
We call an element $x\in B_+$ a \emph{monomial} if $\pi_B^{-1}(x)$ contains a pure tensor.

We need a monomial $K$-basis for $B$. Let $R'\subset \la X\ra$ be the set
$$
\{a_nb_nc_n,c_0a_0\}\cup
\{c_{i+1}a_{i+1}b_{i+1},  b_{i+1}c_{i+1}a_{i+1},  c_ic_{i+1}, b_{i+1}a_i\}_{0\le i<n}
$$
and
$$
\la R'\ra = \{ AWB\ |\ A, B\in \la X\ra, W\in R'\}.
$$
The following proposition is a straightforward application of Bergman's Diamond Lemma \cite{Berg}. We order pure tensors by degree-lexicographic order with the basis elements ordered 
$a_i<b_i<c_i<a_{i+1}<b_{i+1}<c_{i+1}$ for $0\le i\le n-1$.

\begin{prop}
\label{basis}
The image under $\pi_B$ of $\la X\ra-\la R'\ra$, the set of tensors  
which do not contain any element of $R'$, gives a monomial $K$-basis for $B$.
\end{prop}

\begin{rmk}
In Section 5 of \cite{CS}, Cassidy and Shelton give a simple algorithm for determining if a monomial algebra is $\K_2$. The algorithm can be used to show the algebra $T(V)/\la R'\ra$ is $\K_2$. Proposition \ref{basis} implies that $R$ is an \emph{essential Gr\"{o}bner basis} for $I$ (see \cite{Phan}). By Theorem 3.13 in \cite{Phan}, the fact that $T(V)/\la R'\ra$ is $\K_2$ implies $B$ is $\K_2$. Though this method of proving $B$ is $\K_2$ is quite easy, we do not present the details here. A minimal resolution of the trivial module is central to our calculation of $A_{\infty}$-structure, so we prove that $B$ is $\K_2$ using the matrix criterion of Theorem \ref{K2criterion}.
\end{rmk}

The basis provided by Proposition \ref{basis} determines a vector space splitting $\rho_B:B\rightarrow T(V)$ such that $T(V)=\rho_B(B)\oplus I$ by mapping a basis element to its pre-image in $\la X\ra-\la R'\ra$. If $x\in B$, we write $\hat{x}$ for the image of $x$ under $\rho_B$.  If $w\in T(V)$, there exist unique $w^c\in\rho_B(B)$ and $w^r\in I$ such that $w=w^c+w^r$. We call $w^c$ the \emph{canonical form} of $w$, and if $w=w^c$ we say $w$ is
 \emph{in canonical form}. We say $w$ is \emph{reducible} if it is not in canonical form. 
 
As a consequence of Proposition \ref{basis}, we show that left annihilators of monomials are monomial left  ideals.
 
\begin{prop}
\label{Annihilators}
Let $x\in B$ be a nonzero monomial of degree $d'$ and let $w\in l.ann_B(x)$ be homogeneous of degree $d$. If $\hat{w}=\sum_{t=1}^k \alpha_t w_t$ where the $w_t\in \la X \ra_d$ are distinct and $\alpha_t\in K^*$, then $\pi_B(w_t\hat{x})=0$ for $1\le t\le k$.
\end{prop}

\begin{proof}
Let $\hat{x}=x_1\cdots x_{d'}\in \la X \ra_{d'}$ and $w_t=w_{t,1}\cdots w_{t,d}\in \la X \ra_{d}$.
It suffices to consider the case where $\pi_B(w_t\hat{x})$ is nonzero for every $t$. Since $\hat{w}\hat{x}\in I$, $w_t\hat{x}$ is reducible for some $t$. Reordering if necessary, there is an index $j$ such that $w_t\hat{x}$ is reducible for $1\le t\le j$ and in canonical form for $j< t\le k$.

Fix an index $t$ such that $1 \leq t \leq j$. Since $\pi_B(w_t\hat{x})\neq 0$, Proposition \ref{basis} implies that $w_t\hat{x}$ must contain $c_ia_ib_i$ for some $1\le i\le n$. Since $\hat{w}$ and $\hat{x}$ are in canonical form, either $w_{t,d-1}w_{t,d}x_1=c_ia_ib_i$ or $w_{t,d}x_1x_2=c_ia_ib_i$. In the first case, define
$$
y_t=\alpha_t w_{t,1}\cdots w_{t,d-2}(a_{i-1}b_{i-1}c_{i-1})x_2\cdots x_{d'}
$$ 
and in the second case,
$$
y_t= \alpha_t w_{t,1}\cdots w_{t,d-1}(a_{i-1}b_{i-1}c_{i-1})x_3\cdots x_{d'}.
$$
Clearly, $\a_t w_t\hat{x}+y_t\in I$ so $z=\sum_{t=1}^j -y_t+\sum_{t=j+1}^k \a_t w_t\hat{x}\in I$. We claim that $z$ is in canonical form, and for this it suffices to check that $y_t$ is in canonical form. In the case $x_1=b_i$, Proposition \ref{basis} implies $y_t$ is reducible only if $w_{t,d-2}= c_{i-1}$ or $w_{t,d-2}=b_i$ or $x_2=a_{i-1}$ or $x_2 = c_i$. If $x_2=c_i$, then $\pi_B(y_t)=0$ which implies 
$\pi_B(w_t\hat{x})=\pi_B(w_t\hat{x}+y_t)=0$, a contradiction. The other three cases all lead to the contradiction that $\hat{w}$ or $\hat{x}$ is reducible. The case $x_1x_2=a_ib_i$ is the same as the case $x_1 = b_i$, with a shift in index. We conclude $z$ is in canonical form. Since $z \in I$ we know $z=0$ in $T(V)$. We note that all of the pure tensors appearing in $z=\sum_{t=1}^j -y_t+\sum_{t=j+1}^k \a_t w_t\hat{x}$ are distinct. Therefore $\a_t = 0$ for all $1 \leq t \leq k$, which is a contradiction.

\end{proof}

Henceforth, we identify the basis vectors $a_i$, $b_i$, $c_i$ with their images in $B$. 
The following lemma is immediate from Proposition \ref{Annihilators} and the presentation of the algebra.

\begin{lemma} 
\label{MonomialAnnihilators}
For $0\le i\le n$, the element $b_i\in B$ is left regular. 
For $1 \le i \le n-1$, 
$$\text{l.ann}_B(b_ic_i) = Bc_ia_i,\quad
\text{l.ann}_{B}(a_i)=Bb_ic_i+Bb_{i+1},\quad \text{and} $$
$$\text{l.ann}_B(c_ia_i)=Bb_i+Bc_{i-1}.$$
Furthermore
$$\text{l.ann}_B(b_nc_n) = Ba_n, \quad\text{l.ann}_B(c_n)=Ba_nb_n+Bc_{n-1}, \text{ and}$$
$$\text{l.ann}_B(a_0)=\text{l.ann}_B(a_0b_0)=Bb_1+Bc_0.
$$
\end{lemma}

\section{A Minimal Resolution of $_BK$}

In this section we construct a minimal resolution of the trivial $B$-module $_B K$. We then use a criterion of Cassidy and Shelton to prove $B$ is a $\K_2$-algebra.

For $d \in \Z$ let $B(d)$ denote the $B$-module $B$ with grading shifted by $d$, that is $B(d)_k = B_{k+d}$. If $\bar{d}=(d_1,\ldots,d_r)\in\Z^r$ we define 
$$
B(\bar{d})=B(d_1,d_2,\ldots,d_r) = B(d_1)\oplus B(d_2)\oplus\cdots\oplus B(d_{r}).
$$ 

If $Q=B(d_1,\ldots,d_r)$ and $Q'=B(D_1,\ldots,D_{r'})$ are graded free left $B$-modules and $M=(m_{i,j})$ is an $r \times r'$ matrix of homogeneous elements of $B$ such that $\deg m_{i,j}=D_j-d_i$, then right multiplication by $M$ defines a degree 0 homomorphism $f:Q\rightarrow Q'$. We denote this homomorphism $Q\xrightarrow{M} Q'$, and for convenience refer to both the matrix and the homomorphism it defines as $M$. A graded free resolution 
$$ \cdots \rightarrow Q_n \xrightarrow{M_n} Q_{n-1} \rightarrow \cdots \rightarrow Q_0$$ 
of the $B$-module $N$ is {\it minimal} if $\im(M_n) \subseteq B_+ Q_{n-1}$ for every $n \ge 1$. Equivalently, each entry of the matrix representation of $M_n$ is an element of $B_+$. 

\begin{lemma} 
\label{PreAi}
For $1 \leq i \leq n-1$ the sequence 
$$
B(-6,-5)\xrightarrow{\begin{pmatrix} b_ic_i\\ b_{i+1}\\\end{pmatrix}}
B(-4)\xrightarrow{\begin{pmatrix} a_i & c_{i+1}a_{i+1}\\ \end{pmatrix} }
B(-3,-2)\xrightarrow{\begin{pmatrix} b_ic_i\\ b_{i+1}\\ \end{pmatrix}} 
B(-1)
$$ 
of graded free $B$-modules is exact at $B(-4)$ and $B(-3,-2)$.
\end{lemma}

\begin{proof} 
The sequence is clearly a complex. Exactness at $B(-4)$ is clear from Lemma \ref{MonomialAnnihilators}. To prove exactness at $B(-3,-2)$, let $w$, $x\in B$ and suppose 
${\pi_B(\hat{w}b_ic_i+\hat{x}b_{i+1})=0}$. Let $\hat{w}=w_1+w_2$ and $\hat{x}=x_1+x_2$ where $w_1$, $w_2$, $x_1$, and $x_2\in T(V)$ (all in canonical form) are such that $w_1b_ic_i$ and $x_1b_{i+1}$ are in canonical form, and all pure tensors in $w_2b_ic_i$ and $x_2b_{i+1}$ are reducible. 

By Proposition \ref{basis}, $w_2=w'c_ia_i$ and $x_2=x'c_{i+1}a_{i+1}$ for some $w', x' \in T(V)$. So ${\pi_B(w_2b_ic_i)= \pi_B(w'c_ia_ib_i c_i) = \pi_B(-w'a_{i-1}b_{i-1}c_{i-1}c_i) =0}.$
Since $x_2$ is in canonical form, no pure tensor in $x'$ can end in $b_{i+1}$ or $c_i$. Now consider 
\begin{align*}
0 &= \pi_B(w_1 b_i c_i+x_1b_{i+1}+x_2b_{i+1}) \\
   &= \pi_B(w_1 b_i c_i+x_1b_{i+1}-x'a_ib_ic_i).
\end{align*}
Thus $z = w_1 b_i c_i+x_1b_{i+1}-x'a_ib_ic_i$ is in $I$. Since $z$ is also in canonical form we know $z=0$. Therefore $x_1 b_{i+1} = x'a_i b_i c_i - w_1 b_i c_i$ in $T(V)$ and it follows that $x_1=0$. 

We have $\pi_B(\hat{w}b_ic_i-x'a_ib_ic_i)=0$, thus $\pi_B(\hat{w}-x'a_i)\in\text{l.ann}_B(b_ic_i)$. By Lemma \ref{MonomialAnnihilators}, $w-\pi_B(x')a_i\in Bc_ia_i$, so there exists $z\in B$ such that $w=zc_ia_i+\pi_B(x')a_i$. Therefore
\begin{align*}
(w\ \ x )&= ( zc_ia_i+\pi_B(x')a_i\ \ \pi_B(x')c_{i+1}a_{i+1} )\\
&= (zc_i+\pi_B(x')) (a_i\ \ c_{i+1}a_{i+1} ).
\end{align*}
This proves exactness at $B(-3,-2)$.

\end{proof}

\begin{lemma}
\label{ResolutionOfAi}
The periodic complex
$$\cdots
B(-6)\xrightarrow{(b_nc_n)}
B(-4)\xrightarrow{(a_n)}
B(-3)\xrightarrow{(b_nc_n)}
B(-1)
$$ is a minimal graded free resolution of $Ba_n$.

If $1\le i\le n-1$, the periodic complex
$$\cdots\rightarrow
B(-3t-4)\xrightarrow{\begin{pmatrix} a_i & c_{i+1}a_{i+1}\\ \end{pmatrix} }B(-3t-3,-3t-2)\xrightarrow{
\begin{pmatrix} b_ic_i\\ b_{i+1}\\ \end{pmatrix}} 
$$
$$
\cdots 
\rightarrow B(-4)
\xrightarrow{\begin{pmatrix} a_i & c_{i+1}a_{i+1}\\ \end{pmatrix} }B(-3,-2)
\xrightarrow{\begin{pmatrix} b_ic_i\\ b_{i+1}\\ \end{pmatrix}}B(-1)$$
of free left $B$-modules is a  minimal graded free resolution of $Ba_i$.
\end{lemma}

\begin{proof}
The resolution of $Ba_n$ is immediate from Lemma \ref{MonomialAnnihilators}. 

Let $1\le i\le n-1$. By Lemma \ref{MonomialAnnihilators}, $b_ic_i$ and $b_{i+1}$ generate the left annihilator of $a_i$. Exactness in higher degrees follows from Lemma \ref{PreAi} and a degree shift.

\end{proof}

The complexes of Lemma \ref{ResolutionOfAi} will be direct summands of our minimal resolution of $_BK$. We denote these complexes by $P^i_{\bullet}$ where $P^i_1=B(-1)$ for $1 \leq i \leq n$. 
The other summands of our resolution are built inductively. The following lemma provides the base cases for induction.

\begin{lemma}
\label{InductionBase}\ 
\begin{enumerate}
\item If $w$, $x\in B$ and $wc_ia_ib_i+xc_{i-1}=0$ for some $1\le i\le n$, then there exist $w',\ w'',\ x'\in B$ such that 
$$
w=w'+w''c_{i-1}\text{ and }{x=w'a_{i-1}b_{i-1}+x'c_{i-2}}.
$$
\item The complexes
$$
B(-4,-3)\xrightarrow{\begin{pmatrix} c_ia_i & a_{i-1}b_{i-1}\\ 0 & c_{i-2}\\ \end{pmatrix}}
B(-2,-2)\xrightarrow{\begin{pmatrix} b_i\\ c_{i-1} \end{pmatrix}}
B(-1)
$$ 
are exact at $B(-2,-2)$ for $1\le i\le n$.
\item The complex
$$
B(-4,-3)\xrightarrow{\begin{pmatrix} c_n & a_{n-1}b_{n-1}\\ 0 & c_{n-2}\\ \end{pmatrix}}
B(-3,-2)\xrightarrow{\begin{pmatrix} a_nb_n\\ c_{n-1} \end{pmatrix}}
B(-1)
$$ 
is exact at $B(-3,-2)$.
\end{enumerate}
\end{lemma}

\begin{proof}
For (1), we assume $w$, $x\in B$ and $wc_ia_ib_i+xc_{i-1}=0$ for some $i$,
$1\le i\le n$. Since $c_ia_ib_i=-a_{i-1}b_{i-1}c_{i-1}$, we have 
$${\pi_B(\hat{x}c_{i-1}-\hat{w}a_{i-1}b_{i-1}c_{i-1})=0}.$$  Let $\hat{w}=w_1+w_2$ and 
$\hat{x}=x_1+x_2$ where $w_1$, $w_2$, $x_1$, and $x_2\in T(V)$ are such that $w_1a_{i-1}b_{i-1}c_{i-1}$ and $x_1c_{i-1}$ are in canonical form, and all pure tensors in $w_2a_{i-1}b_{i-1}c_{i-1}$ and $x_2c_{i-1}$ are reducible. 

By Proposition \ref{basis}, there exist $y'$, $y''$, and $z'\in T(V)$ so that $w_2=y'b_i+y''c_{i-1}$ and $x_2=z'c_{i-2}$. (If $i=1$, then $z'=0$.) Therefore we have $\pi_B(w_2a_{i-1}b_{i-1}c_{i-1})=0$ and 
$\pi_B(x_2c_{i-1})=0$.

We have $\pi_B(x_1c_{i-1}-w_1a_{i-1}b_{i-1}c_{i-1})=0$, and since all terms are in canonical form, $x_1=w_1a_{i-1}b_{i-1}$. Since ${b_ia_{i-1}=0}$, we may write $x_1=(w_1+y'b_i)a_{i-1}b_{i-1}$. The result follows by setting $w'=\pi_B(w_1+y'b_i)$, $w''=\pi_B(y'')$, and $x'=\pi_B(z')$.

For (2), suppose $w, x\in B$ such that $wb_i+xc_{i-1}=0$. Consider $\hat x c_{i-1}$ and notice that no pure tensor in its canonical form can end in $b_i$. So every pure tensor in 
$\hat{w}b_i$ is reducible. By Proposition \ref{basis}, there exists $w'\in B$ such that 
$w=w'c_ia_i$. By (1), there exist $w'', w''', x' \in B$ such that $$w' = w''+w'''c_{i-1} \text{ and } x = w''a_{i-1} b_{i-1} + x'c_{i-2}.$$ Now note that $x = w'a_{i-1}b_{i-1} + (x'+w'''a_{i-2}b_{i-2})c_{i-2}$. The result follows.

For (3), suppose $w, x\in B$ such that $wa_nb_n+xc_{n-1}=0$. Part (2) implies $wa_n=w'c_na_n$ and $x=w'a_{n-1}b_{n-1}+x'c_{n-2}$ for some $w', x' \in B$. From the proof of (2), $\hat{w'}c_na_n$ may be assumed to be in canonical form. Let $\hat{w}=w_1+w_2$ where $w_1a_n$ is in canonical form and all pure tensors of $w_2a_n$ are reducible. By Proposition \ref{basis}, there exists $y'\in B_n$ such that $w_2=y'b_nc_n$. Thus $\pi_B(w_2a_n)=0$ and we have $\pi_B(w_1a_n-\hat{w'}c_na_n)=0$. Since both terms are in canonical form, $w_1=\hat{w'}c_n$. So 
$$
w=\pi_B(\hat{w'}+y'b_n)c_n\text{ and } x=\pi_B(\hat{w'}+y'b_n)a_{n-1}b_{n-1}+\pi_B(x')c_{n-2}.
$$ 
The result follows.

\end{proof}

The following lemma is clear by combining Lemma \ref{MonomialAnnihilators} with Lemma \ref{InductionBase}(2).

\begin{lemma}
\label{ResolutionA0}
The periodic complex
$$
\cdots\rightarrow
B(-3t-4)\xrightarrow{\begin{pmatrix} a_{0}b_{0} & c_1a_1\\ \end{pmatrix} }
B(-3t-2,-3t-2)\xrightarrow{\begin{pmatrix} c_0\\ b_1\\ \end{pmatrix}}
$$
$$
\cdots\rightarrow B(-4)\xrightarrow{\begin{pmatrix}  a_0b_0 & c_1a_1\\ \end{pmatrix} }
B(-2,-2)\xrightarrow{\begin{pmatrix} c_0\\ b_1\\ \end{pmatrix}}
B(-1)
$$
of free left $B$-modules is a minimal graded free resolution of $Ba_0$.
\end{lemma}

We denote the complex of Lemma \ref{ResolutionA0} by $P^0_{\bullet}$ where $P^0_1=B(-1)$.

To simplify the exposition from this point, we introduce an operation which we use to inductively build matrices in our minimal resolution of $_BK$.

If $M=(m_{i,j})$ is an $a\times b$ matrix and $N=(n_{i,j})$ is a $c\times d$ matrix, we define the \emph{star product}  $M\star N$ to be the $(a+c-1)\times (b+d)$ matrix 
$$
\begin{pmatrix} 
m_{1,1} & \cdots & m_{1,b} &               &            &              \\
\vdots    &             &   \vdots  &               &  0        &               \\
m_{a,1} & \cdots & m_{a,b} & n_{1,1} & \cdots & n_{1,d} \\
               &   0       &                & \vdots    &            &\vdots    \\
               &            &                & n_{c,1}  & \cdots & n_{c,d} \\
\end{pmatrix}.
$$
We note that this product is associative and we define 
$$
M\star_{j=1}^p N_j=M\star N_1\star \cdots  \star N_p.
$$ 

For $1\le i\le n$ let $\Gamma_i=\begin{pmatrix} a_ib_i\\ c_{i-1}\\ \end{pmatrix}$ and let $T_i = \begin{pmatrix} b_i \\ c_{i-1} \\ \end{pmatrix}$. Let
$\Gamma_0=(a_0b_0)$.  It will be convenient to define $M\star \Gamma_i =M$ for $i<0$ and $(\ ) \star \Gamma_i=\Gamma_i$, where $( \ )$ denotes the empty matrix. 

The next lemma is the key to making inductive arguments in later proofs. 

\begin{lemma}
\label{InductionStep}
Let $r_1$, $r_2$, $r_3$ be positive integers and $\bar{d_1}\in\Z^{r_1}$, ${\bar{d_2}\in\Z^{r_2}}$, 
$\bar{d_3}\in\Z^{r_3}$. Let $d_{j, k}$ denote the $k$th component of $\bar{d_j}$. Let $i$ be a positive integer. Let $M$ and $N$ be matrices of homogeneous elements of $B$ such that 
\begin{equation*}
\quad B(\bar{d_3})\xrightarrow{N\star\Gamma_i} B(\bar{d_2})\xrightarrow{M} B(\bar{d_1})
\end{equation*} 
is an exact sequence of degree zero homomorphisms.
\begin{enumerate}

\item If $i = 1$, then the sequence
$$B(\bar{d_3}) \xrightarrow{N\star \Gamma_1}
B(\bar{d_2}) \xrightarrow{M\star \Gamma_0}
B(\bar{d_1}) \oplus B(d_{2, r_2}+2)
$$
is exact at $B(\bar{d_2})$.

\item If $i = 2$, then the sequence
$$B(\bar{d_3}) \xrightarrow{N\star \Gamma_2 \star \Gamma_0}
B(\bar{d_2}) \oplus B(d_{3, r_3}+2) \xrightarrow{M\star \Gamma_1}
B(\bar{d_1}) \oplus B(d_{3, r_3}+3)
$$
is exact at $B(\bar{d_2}) \oplus B(d_{3, r_3})$.

\item If $i > 2$, then for all $t \in \Z$ the sequence
$$
B(\bar{d_3})\oplus B(t-2)\xrightarrow{N\star \Gamma_{i}\star\Gamma_{i-2}} 
B(\bar{d_2})\oplus B(t-1)\xrightarrow{M\star \Gamma_{i-1}} 
B(\bar{d_1})\oplus B(t)
$$ 
is exact at $B(\bar{d_2})\oplus B(t-1)$.

\item  If $i > 2$ and 
\begin{equation*}
\quad B(\bar{d_3})\xrightarrow{N\star T_i} B(\bar{d_2})\xrightarrow{M} B(\bar{d_1})
\end{equation*} 
is exact at $B(\bar{d_2})$, 
then the sequence 
$$
B(\bar{d_3})\oplus B(t-2)\xrightarrow{N\star T_i\star\Gamma_{i-2}} 
B(\bar{d_2})\oplus B(t-1)\xrightarrow{M\star \Gamma_{i-1}} 
B(\bar{d_1})\oplus B(t)
$$ is exact at $B(\bar{d_2})\oplus B(t-1)$.
\end{enumerate}
\end{lemma}

\begin{proof} The statements (1) and (2) follow immediately from Lemma \ref{MonomialAnnihilators}.
So we assume $i>2$.

For (3), the hypothesis implies that the rows of $N'=N\star \G_i\star
\begin{pmatrix} 0\\  1\\ \end{pmatrix}$ generate the kernel of $\begin{pmatrix} M\\ 0\\ \end{pmatrix}$.   
To compute the kernel of  $M\star\Gamma_{i-1}$, it suffices to determine which elements of the submodule generated by the rows of $N'$ also left annihilate the last column of 
$M\star\Gamma_{i-1}$.

Since $b_{i}a_{i-1}=0$, the last two rows of $N'$ are the only rows not in the kernel of
$M\star \Gamma_{i-1}$. A row  $(\ 0\ \cdots\ 0\ wc_{i-1}\ x\ )$ is in the kernel of 
$M\star \Gamma_{i-1}$ if and only if $wc_{i-1}a_{i-1}b_{i-1}+xc_{i-2}=0$. By Lemma \ref{InductionBase}(1), this occurs if and only if $(\ 0\ \cdots\ 0\ wc_{i-1}\ x\ )$ is a $B$-linear combination of $(\ 0\ \cdots\ 0\ \ c_{i-1}\ \ a_{i-2}b_{i-2}\ )$ and $(\ 0\ \cdots\ 0\ \ 0\ \ c_{i-3}\ )$, which are the last two rows of $N\star\Gamma_i\star\Gamma_{i-2}$. Since all other rows of $N'$ and 
$N\star\Gamma_i\star\Gamma_{i-2}$ are equal, the proof of (3) is complete.

The proof of (4) is the same as the proof of (3).

\end{proof}

To help keep track of the ranks and degree shifts of free modules which appear in later resolutions we introduce the following notations. If $i, j \in \Z$ and $i\le j$, we denote $(i, i+1, \ldots, j-1, j)$ by $[i,j]$. 

If $i$ is even,  for $j\in\N$, let 
$$
\bar{d}_i(j)=
\begin{cases} 
[-\lfloor 3j/2 \rfloor, -j  ]          & j\le i\\
[-\lfloor 3j/2 \rfloor, - \lceil 3j/2\rceil  +i/2]          & j>i\\
\end{cases}.
$$

If $i$ is odd, for $j\in\N$, let 
$$
\bar{d}_i(j)=
\begin{cases} 
[-\lfloor 3j/2\rfloor, -j]	 & j\le i	 \\
[-\lfloor 3j/2\rfloor, -\lfloor 3j/2\rfloor +\lfloor i/2\rfloor]	& j>i	 \\
\end{cases}.
$$

Figure 1 illustrates the combinatorics of the vector $\bar{d}_i(j)$.

\begin{figure}[ht] 
   \centering
   \includegraphics[width=3.5in]{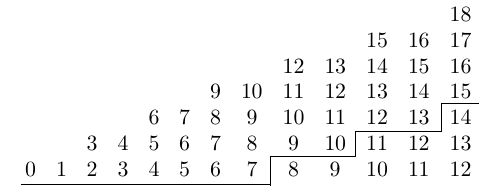} 
   \caption{The degree vector $\bar{d}_i(j)$ consists of negatives of values above the line in column $j$. In this figure $i=7$.}
   \label{fig1}
\end{figure}

We need a bit more notation.
Let $0\le p\le n$ such that $n\equiv p\pmod{2}$. Let
$$
R_{p}=\Gamma_{n}\star_{j=1}^{(n-p)/2}\Gamma_{n-2j} \text{ and }
U_{p-1}=
\begin{pmatrix} 
c_n\\ 
\end{pmatrix}
\star_{j=0}^{(n-p)/2} \Gamma_{n-2j-1}.
$$

\begin{lemma}
\label{ResolutionCn} 
\begin{enumerate}

\item If $n$ is even, the eventually periodic sequence
$$
\cdots 
\xrightarrow{U_1} B(\bar{d}_{n+1}(n+2))
\xrightarrow{R_0} B(\bar{d}_{n+1}(n+1))
\xrightarrow{U_1} B(\bar{d}_{n+1}(n))
\xrightarrow{R_2}
\cdots
$$
$$
\xrightarrow{U_{n-3}}B(\bar{d}_{n+1}(4))
\xrightarrow{R_{n-2}} B(\bar{d}_{n+1}(3))
\xrightarrow{U_{n-1}}B(\bar{d}_{n+1}(2))
\xrightarrow{R_n}B(\bar{d}_{n+1}(1))
$$ 
is a minimal graded free resolution of $Bc_n$.

\item If $n$ is odd, the eventually periodic sequence
$$
\cdots 
\xrightarrow{R_1} B(\bar{d}_{n+1}(n+2))
\xrightarrow{U_0} B(\bar{d}_{n+1}(n+1))
\xrightarrow{R_1} B(\bar{d}_{n+1}(n))
\xrightarrow{U_2} 
\cdots
$$
$$
\xrightarrow{U_{n-3}}B(\bar{d}_{n+1}(4))
\xrightarrow{R_{n-2}} B(\bar{d}_{n+1}(3))
\xrightarrow{U_{n-1}}B(\bar{d}_{n+1}(2))
\xrightarrow{R_n}B(\bar{d}_{n+1}(1))
$$ 
is a minimal graded free resolution of $Bc_n$.
\end{enumerate}
\end{lemma}

\begin{proof}
It is clear that the homomorphisms are degree zero by inspecting the degrees of entries in $U_{p-1}$ and $R_p$. We note that $R_n = \G_n$. Since $\text{l.ann}_B(c_n) = Ba_nb_n + Bc_{n-1}$ we see the resolutions have the correct first terms to be minimal resolutions of $Bc_n$. 

We will prove (1). If $n=0$, the result follows from Lemma 2.3.  So we assume $n>0$ and $n$ is even. Exactness at $B(\bar{d}_{n+1}(2))$ is Lemma 3.3 (3). For exactness at $B(\bar{d}_{n+1}(3))$, we apply Lemma 3.5 with $i = n$, $M = \begin{pmatrix} c_n \\ \end{pmatrix}$ and $N$ the empty matrix. Similarly, exactness at $B(\bar{d}_{n+1}(4))$ follows from exactness at $B(\bar{d}_{n+1}(2))$ by applying Lemma 3.5 with $i = n-1$, $M = R_n$, and $N = (c_n)$. We remark that exactness at $B(\bar{d}_{n+1}(n+1))$ follows from exactness at $B(\bar{d}_{n+1}(n-1))$ by applying Lemma 3.5 (2). Exactness at $B(\bar{d}_{n+1}(n+2))$ follows from exactness at $B(\bar{d}_{n+1}(n))$ by applying Lemma 3.5 (1). By induction, we conclude the sequence in (1) is exact.

The proof of (2) is similar and is omitted.

\end{proof}

The resolution of Lemma \ref{ResolutionCn} is the second piece of our resolution of $_BK$. We denote this complex by $C_{\bullet}$ where $C_1=B(-1)$. 

For integers $i, p, n$ such that $1 \le i \le n$, $0\le p\le i$, and $i\equiv p\pmod{2}$, define
 $$
 S_{i,p-1}=
\begin{pmatrix} 
c_ia_i 
\end{pmatrix} 
\star_{j=0}^{(i-p)/2} \Gamma_{i-2j-1} 
\text{ and }
T_{i,p}=
T_i
\star_{j=1}^{(i-p)/2} \Gamma_{i-2j},
$$
where $T_i = \begin{pmatrix} b_i \\ c_{i-1} \\ \end{pmatrix}$.

\begin{lemma}\ 
\label{ResolutionBi}
\begin{enumerate}

\item If $1\le i\le n$ and $i$ is even, then the eventually periodic sequence
$$
\cdots
\xrightarrow{S_{i,1}} B(-3i/2-1,\bar{d}_{i}(i+1))
\xrightarrow{T_{i,0}} B(\bar{d}_i(i))
\xrightarrow{S_{i,1}}
$$
$$
\cdots
\xrightarrow{S_{i,i-3}} 
B(-4,\bar{d}_i(3))
\xrightarrow{T_{i,i-2}} B(\bar{d}_i(2))
\xrightarrow{S_{i,i-1}} B(-1,\bar{d}_i(1))
$$
is a minimal graded free resolution of $Bb_i+Bc_{i-1}$.

\item If $1\le i\le n$ and $i$ is odd, the eventually periodic sequence
$$
\cdots
\xrightarrow{T_{i,1}} B(\bar{d}_{i}(i+1))
\xrightarrow{S_{i,0}} B(-(3i-1)/2,\bar{d}_i(i))
\xrightarrow{T_{i,1}} B(\bar{d}_{i}(i-1))
\xrightarrow{S_{i,2}}
$$
$$
\cdots
\xrightarrow{S_{i,i-3}} B(-4,\bar{d}_i(3))
\xrightarrow{T_{i,i-2}} B(\bar{d}_i(2))
\xrightarrow{S_{i,i-1}} B(-1,\bar{d}_i(1))
$$
is a minimal graded free resolution of $Bb_i+Bc_{i-1}$.
\end{enumerate}
\end{lemma}

\begin{proof} We will prove (2). The case $i=1$ follows immediately from Lemma 3.4. Suppose that $i>2$ and $i$ is odd. Lemma 3.3 (2) shows the first map $S_{i, i-1}$ is the start of a minimal resolution of $Bb_i+Bc_{i-1}$. Exactness in higher degrees follows by induction and the remark in the paragraph following the proof of Lemma 3.5.

The proof of (1) is similar and is omitted.

\end{proof}

For $1\le i\le n$, we denote the complexes of Lemma \ref{ResolutionBi} by $Q^i_{\bullet}$ with $Q^i_1=B(-1,-1).$ Denote by $Q^0_{\bullet}$ the complex $0\rightarrow B(-1)$ which is a minimal graded free resolution of $Bb_0$ by Lemma \ref{MonomialAnnihilators}. We take $Q^0_1=B(-1)$.

Define the chain complex $\widetilde{Q_{\bullet}}$ by
$$
\widetilde{Q_{\bullet}} =
P^0_{\bullet}\oplus P^1_{\bullet}\oplus\cdots\oplus P^n_{\bullet}
\oplus C_{\bullet}
\oplus Q^n_{\bullet}\oplus\cdots\oplus Q^1_{\bullet}\oplus Q^0_{\bullet}.
$$ 
Let $M_1=(\ a_0\ \ a_1\ \cdots\ a_n\ \ c_n\ \ b_n\ \cdots c_0\ \ b_0\ )^T$. Let $M_d$ be the matrix of 
$\widetilde{Q}_d\rightarrow\widetilde{Q}_{d-1}$. Denote by $\hat{M_d}$ the matrix with entries in $T(V)$  such that $(\hat{M_d})_{i,j}=\rho_B((M_d)_{i,j})$.

\begin{thm} 
\label{ResolutionK}
The complex $\widetilde{Q}_{\bullet}\xrightarrow{M_1}B $ is a minimal graded free resolution of the trivial left $B$-module $_BK$
\end{thm}

\begin{proof}
By Lemmas \ref{ResolutionOfAi}, \ref{ResolutionA0}, \ref{ResolutionCn}, and \ref{ResolutionBi}, it is enough to check exactness at $B$ and at $\widetilde{Q}_1$.
The image of $M_1$ in $B$ is clearly the augmentation ideal $B_+$. The complex is exact at $\widetilde{Q}_1$ since the entries of $\hat{M_2}\hat{M_1}$ give the set $R$ of relations of $B$. 

\end{proof}

Next we show the algebra $B$ is a $\K_2$-algebra using the criterion established by Cassidy and Shelton in \cite{CS}. 
Put $I'=V \tensor I + I \tensor V$. An element in $I$ is an {\it essential relation} if its image is nonzero in $I/I'$. For each $d\ge 2$, let $L_d$ be the image of $\hat{M_d}$ modulo the ideal $T(V)_{\ge 2}$. Let $E_d$ be the image of 
$\hat{M_d}\hat{M}_{d-1}$ modulo $I'$. Finally, let $[L_d : E_d]$ be the matrix obtained by concatenating $L_d$ and $E_d$. 

\begin{thm}[\cite{CS}]
\label{K2criterion}
The algebra $B$ is a $\K_2$-algebra if and only if for all $d>2$, $\widetilde{Q}_d$ is a finitely generated $B$-module and the rows of $[L_d : E_d]$ are independent over $K$.
\end{thm}

\begin{thm} 
\label{K2}
For any $n \in \N$, $B$ is a $\K_2$-algebra.
\end{thm}

\begin{proof} 
Let $d > 2$. It is clear that $\widetilde{Q}_d$ is finitely generated. To see that the rows of $[L_d : E_d]$ are independent, it suffices to check the condition on the blocks of $\hat{M}_d$ and $\hat{M}_{d-1}$. The blocks $U_p$ and $T_{i,p}$ have exactly one linear term in each row and no two are in the same column except for the upper left corner of $T_{i, p}$ which is $\begin{pmatrix} b_i  & c_{i-1} \end{pmatrix}^T$ and $b_i, c_{i-1}$ are independent over $K$. So the condition holds for these blocks. 

The blocks  $(\ a_n\ )$, $(\ a_i\ \ c_{i+1}a_{i+1}\ )$, and $(\ b_1\ \ c_0\ )^T$  contain linear terms in each row and the rows are independent, so they satisfy the condition. The block 
$(\ c_1a_1\ \ a_0b_0\ )$ does not contain a linear term, but the corresponding block of $E_d$ is the essential relation $c_1a_1b_1+a_0b_0c_0$. The blocks $(\ b_ic_i\ \ b_{i+1}\ )^T$ have a linear term in the second row, and the first row of the corresponding block of $E_d$ contains the essential relation $b_ic_ia_i$. 

The blocks $S_{i,p}$ and $R_p$ have one linear term in each row except the first, and no two are in the same column, so it is enough to check that the first row of the corresponding block of $E_d$ is nonzero. A direct calculation shows that respectively the rows contain the essential relations
$c_ia_ib_i+a_{i-1}b_{i-1}c_{i-1}$ and $a_nb_nc_n$.

\end{proof}

\section{$A_{\infty}$-algebra structures from resolutions}

Let $A$ be a $K$-algebra and let $(Q_{\bullet},d_{\bullet})$ be a minimal graded projective resolution of $_AK$ by $A$-modules with $Q_0=A$. For $n\in\Z$, a {\it degree $n$ endomorphism} of $(Q_{\bullet},d_{\bullet})$ is a collection of degree zero homomorphisms of graded $A$-modules $\{f_j:Q_j\rightarrow Q_{j+n}\ |\ j\in\Z \}$. Note that $f_j=0$ for $j<\max\{0,-n\}$.

Let $U=\End_A(Q_{\bullet})$ be the differential graded endomorphism algebra of $(Q_{\bullet},d_{\bullet})$ with multiplication given by composition.  We denote the maps in $U$ of degree $-n$ by $U^{n}$ for all $n \in \Z$. 
The differential $\partial$ on $U$ is given on homogeneous elements by $\partial(f) = df-(-1)^{|f|}fd$, where $|f|$ denotes the degree of $f$.  With respect to the endomorphism degree, $\deg(\partial) = 1$. We let $B^{n}$ and $Z^{n}$ respectively denote the set of coboundaries and the set of cocycles in $U^n$. 

\begin{lemma}
\label{Lifting}
Let $g\in Z^n$. If $\im(g_n)\subset (Q_0)_+$, then there exists $f\in U^{n-1}$ with $f_j=0$ for all $j<n$ such that $\partial(f)=g$.
\end{lemma}

\begin{proof}
We define $f$ inductively. Put $f_j=0$ for all $j<n$. Since $Q_{\bullet}$ is a resolution, $\im(d_1)=(Q_0)_+$. Hence $\im(g_n)\subset\im(d_1)$. By graded projectivity of $Q_n$, there exists a degree zero homomorphism of graded $A$-modules $f_n:Q_n\rightarrow Q_1$ such that $g_n=d_1f_n = d_1f_n-(-1)^{n-1}f_{n-1}d_n.$

Fix $j>n$ and assume that for all $k< j$, $f_k$ is defined and $df_k=(-1)^{n-1}f_{k-1}d+g_k$. Then 
\begin{align*}
d[(-1)^{n-1}f_{j-1}d+g_{j}] &= (-1)^{n-1}(df_{j-1} d) + dg_{j}\\
   &= (-1)^{n-1}((-1)^{n-1}f_{j-2}d+g_{j-1})d + dg_{j}\\
   &= (-1)^{n-1}g_{j-1}d + dg_{j} \\
   &= \partial(g)_{j} \\
   &= 0.
\end{align*}	    
Since $Q_{\bullet}$ is a resolution, $\im((-1)^{n-1}f_{j-1}d+g_{j}) \subseteq \im(d_{j+1-n})$. Thus, by graded projectivity of $Q_j$, there exists a degree zero homomorphism of graded $A$-modules $f_{j}:Q_j\rightarrow Q_{j+1-n}$ such that $df_j = (-1)^{n-1}f_{j-1}d+g_j$.

By induction, there exists $f \in U^{n-1}$ such that ${\partial(f) = g}$.

\end{proof}

 The minimality of the resolution implies that $\Hom_A(Q_{\bullet},K)$ is equal to its cohomology. We take $H^*(\Hom_A(Q_{\bullet},K))=\Hom_A(Q_{\bullet},K)$ as our model for the Yoneda algebra $E(A)$. We abuse notation slightly and write $E(A)=\Hom_A(Q_{\bullet},K)$. It will be convenient to view $E(A)$ as a differential graded algebra with trivial differential.

Let $\e:Q_0\rightarrow K$ be the augmentation homomorphism. For every ${n\in\Z}$ define a map $\phi^n:U^n\to \Hom^n_A(Q_{\bullet},K)$ by $\phi^n(f)= \e f_n$. We remark that the map $\phi=\oplus_n \phi^n$ induces a surjective homomorphism $\Phi: H^*(U) \to E(A)$. Using Lemma \ref{Lifting}, it is easy to prove that $\Phi$ is also injective.

We compute the structure of an $A_{\infty}$-algebra on $E(A)$ by specifying the data of a strong deformation retraction (SDR-data) from $U$ to $E(A)$. More precisely, we choose maps 
$i:E(A)\rightarrow U$, $p:U\rightarrow E(A)$, and ${G:U\rightarrow U}$ such that $i$ and $p$ are degree 0 morphisms of complexes and $G$ is a homogeneous $K$-linear map of degree $-1$ such that $pi=1_{E(A)}$ and $1_U-ip=\partial G-G\partial$.

We define a family of homogenous $K$-linear maps $\{\lambda_j : U^{\tensor j} \to U\}_{j \geq 2}$ with $\deg(\l_j) = 2-j$ as follows. There is no map $\lambda_1$, but we define the formal symbol $G\lambda_1$ to mean $-1_U$. The map $\lambda_2$ is the multiplication on $U$. For $n \geq 3$, the map $\l_j$ is defined recursively by
$$\l_j = \l_2 \sum_{ \substack{s+t = j \\
s, t \geq 1}}
(-1)^{s+1}(G\l_s \tensor G\l_t).$$

The following theorem due to Merkulov provides an $A_{\infty}$-structure on $E(A)$. Merkulov's theorem applies to the subcomplex $i(E(A))$ of $U$ and gives structure maps $m_j=ip\lambda_j$.
Since $pi=1_{E(A)}$, we may restate the theorem for $E(A)$ as follows.

\begin{thm}[\cite{Merkulov}] 
\label{Merkulov}
Let $m_1=0$ and for $j \geq 2$, let $m_j = p \l_j i$. Then $(E(A), m_1, m_2, m_3, \ldots)$ is an $A_{\infty}$-algebra satisfying the conditions of Theorem \ref{Kadeishvili}.
\end{thm}	  

There are many choices for $i$ and $G$. We describe a method for defining $i$ and $G$ based on a minimal projective resolution $(Q_{\bullet}, d_{\bullet})$ of the trivial module $_AK$. Our $G$ will additionally satisfy the \emph{side conditions} $G^2=0$, $Gi=0$, and $pG=0$. As observed in \cite{GugLam1}, initial SDR-data can always be altered to satisfy these conditions, but they are corollaries of our construction. We define $p=\phi$.

Next, we define the inclusion map $i$. Assume that a homogeneous $A$-basis $\{u_n^\ell\}$ has been fixed for each ${Q}_n$ and let $\{\m_n^\ell\}$ be the graded dual basis for $E^n(A)$. 

We choose $i: E(A) \to U$ as follows. For $n\ge 0$ let $\m\in E^n(A)$ be a dual basis vector. By graded projectivity, there exists a sequence of degree zero $A$-module homomorphisms $\{i(\m)_k\}_{k \geq n}$ so that the diagram

\begin{diagram} 
&& Q_{n+1} & \rTo &  Q_n & & \\
\cdots&& \dTo^{i(\m)_{n+1}} & & \dTo^{i(\m)_n} & \rdTo^{\m} & \\
&& Q_1 & \rTo & Q_0 & \rTo_{\e} & k \\
\end{diagram}
has commuting squares when $n$ is even and anticommuting squares when $n$ is odd. This ensures that $i(\m)$ is contained in $Z^n$. We choose the map $i(\m)_n$ to be represented by a matrix with scalar entries in the given $A$-basis. We extend $i$ to $E(A)$ by $K$-linearity, and we note that $p i=1_{E(A)}$.

For each integer $n$ define the $K$-vector space $H^n$ by $H^n = i(E^n(A))$ if $n \geq 0$ and $H^n = 0$ if $n < 0$.

\begin{prop} 
\label{BHsplit}
For every $n\in\Z$, $Z^n = B^n \oplus H^n$.
\end{prop}

\begin{proof} Let $n \in \Z$ and let $g \in B^n$. Then there exists $f \in U^{n-1}$ such that $\partial(f) = g$. The minimality condition ensures that $\im(g_n) = \im(\partial(f)_n)$ is contained in $A_+Q_0 = (Q_0)_+ $. Thus, $g_n$ cannot be represented by a matrix with nonzero scalar entries. We conclude that $B^n \cap H^n = 0$.

Now, let $f \in Z^n$. By definition, $ip(f) \in H^n$. To prove $Z^n = B^n \oplus H^n$,
 it suffices to show there exists $g \in U^{n-1}$ such that $\partial(g) = f-ip(f)$. This follows from Lemma \ref{Lifting}. 

\end{proof}

Finally, we define a homogeneous $K$-linear map $G: U \to U$ of degree $-1$. There will be many choices of $G$, but the main properties we want $G$ to satisfy are $\partial G |_{B^n} = 1_{B^n}$ and $G(B^n)\cap Z^{n-1}=0$. 

We start by defining $G$ on the $K$-linear space $B^n$. Let $b$ denote a basis element of $B^n$. The hypotheses of Lemma \ref{Lifting} hold for the map $b$ so there exists a map $f \in U^{n-1}$ such that $\partial(f) = b$. Define $G(b) = f$. 
Extending by $K$-linearity, we have the desired $G$ defined on $B^n$.

 For each integer $n$, let $L^n$ denote the $K$-vector space $G(B^{n+1})$.

\begin{prop} 
\label{LZsplit}
For every $n \in \Z$, $U^n = L^n \oplus Z^n$.
\end{prop}

\begin{proof} Let $n \in \Z$. First suppose $h \in L^n \cap Z^n$. Then $\partial h = 0$ and $h = G(h')$ for some $h' \in B^{n+1}$. So $h' = \partial G(h') = \partial h = 0$. Hence $h = G(h') = G(0) = 0$.

Now let $f \in U^n$. By definition $G \partial(f) \in L^n$. Furthermore we have  $\partial(f - G\partial(f)) = \partial(f) - \partial G \partial(f) = \partial(f) - \partial(f) = 0$. So $f-G\partial(f) \in Z^n$. Since $f = G\partial(f) + (f-G\partial(f))$, we conclude that $U^n = L^n \oplus Z^n$.

\end{proof}

Proposition \ref{BHsplit} and Proposition \ref{LZsplit} allow us to extend $G$ to a $K$-linear map on all of $U$ by defining $G$ to be zero on $H^n\oplus L^n$. It is straightforward to check that $1_U - ip = \partial G + G \partial$ and the side conditions hold.

\section{$A_{\infty}$-algebra structure on $E(B)$}

We return to the setting of Sections 2 and 3.
We wish to compute only a few nonzero multiplications of an $A_{\infty}$-structure on $E(B)$, not the entire structure. Therefore, we define $i$ explicitly only on certain linearly independent elements of $E(B)$. Similarly, we explicitly define $G$ only on certain elements of $U=\End_B(\widetilde{Q}_{\bullet})$. We then apply the results of Section 4 to extend the definition of $i$ to all of $E(B)$ and  the definition of $G$ to all of $U$.

It is standard to identify $E^1(B)$ with the $K$-linear graded dual space $V^*$. For $0\le \ell\le n$, we denote the dual basis vectors of $a_\ell, b_\ell, c_\ell\in V$ by  $\alpha_\ell, \beta_\ell,$ and $ \g_\ell$, respectively. We suppress the $m_2$ notation for products in $E(B)$. For example, we denote $m_2(\b_1\otimes \a_{0})$ by $\b_1\a_0$.

The diagrams below define $i: E(B) \to U$ on $\a_n$, $\b_\ell \a_{\ell-1}$, $\b_0$, and $\g_0 \cdots \g_n$. Recall that we have chosen homogeneous $A$-bases and given the corresponding matrix representations for the resolution $(\widetilde{Q}_{\bullet}, d_{\bullet})$. We write $e_{i, j}$ for the standard matrix unit. We include just enough of the endomorphism to see the general pattern. Finally, to make cleaner statements it will be useful to define 
$$j = \begin{cases} \frac{3n}{2} + 1 &\text{ if $n$ is even } \\
  \frac{5n+3}{2} &\text{ if $n$ is odd } \\
    \end{cases}\quad
    \text{ and }\quad
k = \begin{cases} \frac{5n}{2} + 2 &\text{ if $n$ is even } \\
  \frac{3n+3}{2} &\text{ if $n$ is odd } \\
    \end{cases}.$$	   
We note that the $(j+1, k)$-entry of $M_{n+1}$ is $c_0$ and that this is the lower rightmost position in the block $U_1$ when $n$ is even.  When $n$ is odd, the $(j+1,k)$-entry is the lower rightmost position in $R_1$. The upper leftmost entry of $U_{p-1}$ or $R_p$ is always the $(n+2, 2n+2)$-entry of $M_d$ when $d\ge 3$ is odd. It is the $(2n+2,n+2)$-entry of $M_d$ when $d$ is even.

$i(\a_n)$ for $n>0$ (see Lemma \ref{ResolutionOfAi} and Lemma \ref{ResolutionCn}):
\begin{diagram}
& \widetilde{Q}_4 & \rTo & \widetilde{Q}_3 & \rTo & \widetilde{Q}_2 & \rTo & \widetilde{Q}_1 \\
\cdots & \dTo_{-b_n e_{2n+1, n+2}} & & \dTo_{ e_{n+1, 2n+2}} & & \dTo_{-b_n e_{2n+1, n+2}} & & \dTo_{ e_{n+1, 1}} \\
& \widetilde{Q}_3 & \rTo & \widetilde{Q}_2 & \rTo & \widetilde{Q}_1 & \rTo & \widetilde{Q}_0 \\
\end{diagram}

When $n=0$, we choose $i(a_0)_{t}=(-1)^{t+1}e_{1,2}$ for all $t>1$.

$i(\b_\ell \a_{\ell-1})$, $1 \leq \ell \leq n$ (see Lemma \ref{ResolutionA0} for $\ell=1$ and Lemma \ref{ResolutionOfAi} for $\ell>1$):
\begin{diagram}
& \widetilde{Q}_5 & \rTo & \widetilde{Q}_4 & \rTo & \widetilde{Q}_3 & \rTo & \widetilde{Q}_2 \\
 \cdots & \dTo_{c_\ell e_{\ell, \ell+1}} & & \dTo_{e_{2\ell, 2\ell+1}} & & \dTo_{c_\ell e_{\ell, \ell+1}} & &  \dTo_{e_{2\ell, 1}} \\
& \widetilde{Q}_3 & \rTo & \widetilde{Q}_2 & \rTo & \widetilde{Q}_1 & \rTo & \widetilde{Q}_0 \\
\end{diagram}

$i(\b_0)$:
\begin{diagram}
& \widetilde{Q}_3 & \rTo & \widetilde{Q}_2 &\rTo & \widetilde{Q}_1 \\
\cdots& \dTo_0 & & \dTo_0&    &\dTo_{e_{3n+3, 1}} \\
& \widetilde{Q}_2 & \rTo & \widetilde{Q}_1 & \rTo & \widetilde{Q}_0 \\
\end{diagram}

$i(\g_0 \cdots \g_n)$ (see Lemma \ref{ResolutionCn}):
\begin{diagram}
&\widetilde{Q}_{n+3} & \rTo & \widetilde{Q}_{n+2} & \rTo & \widetilde{Q}_{n+1} \\
\cdots&\dTo_{0} & & \dTo_{(-1)^{n+1} a_0 e_{k, 3n+3}} & & \dTo_{e_{j+1, 1}} \\
&\widetilde{Q}_{2} & \rTo & \widetilde{Q}_{1} & \rTo & \widetilde{Q}_{0} \\
\end{diagram}

We henceforth suppress the inclusion map $i$. From the definitions, it is easy to check that  $\l_2(\b_1 \a_0 \tensor \b_0) = 0$ and  $\l_2(\b_\ell \a_{\ell-1} \tensor \b_{\ell-1} \a_{\ell-2}) = 0$ for $2\le \ell\le n$. It is also easy to see that $\l_2(\b_0 \tensor \g_0 \cdots \g_n) = (-1)^{n+1} a_0 e_{k, 1}$. 

Next, we define values of the homotopy $G$ on certain elements. It is important to observe that the elements on which we define $G$ are linearly independent coboundaries in $U$. Each morphism below is in a different graded component of $U$, so linear independence is clear. To check that $G$ is well-defined, it suffices to observe that $\partial G = 1$ on the element in question. Recall that we define 
$$\l_\ell = \l_2 \sum_{ \substack{s+t = \ell \\
s, t \geq 1}}
(-1)^{s+1}(G\l_s \tensor G\l_t).$$

First, let $G\l_2(\b_0\tensor \g_0\cdots \g_n)=$
\begin{diagram}
& \widetilde{Q}_{n+4} & \rTo & \widetilde{Q}_{n+3} & \rTo &\widetilde{Q}_{n+2}\\
\cdots &\dTo_{(-1)^{n+1}e_{k,1}} & &\dTo_{\begin{matrix} a_1e_{j,2}\\+e_{j+1,1}\\ \end{matrix}} & & \dTo_{(-1)^{n+1}e_{k,1}}\\
  &\widetilde{Q}_3 & \rTo & \widetilde{Q}_2 & \rTo & \widetilde{Q}_1\\
\end{diagram}

Next, since $\lambda_2(\b_1\a_0\tensor\b_0)=0$, we have 
$$\l_{3}(\b_1\a_{0} \tensor \b_0\tensor \g_0\cdots \g_n)=-\l_2(\b_1\a_0\tensor G\l_2(\b_0\tensor \g_0\cdots \g_n)).$$

We choose
$G\l_{3}(\b_1\a_{0} \tensor \b_0\tensor \g_0\cdots \g_n)=$
\begin{diagram}
&   \widetilde{Q}_{n+5} & \rTo & \widetilde{Q}_{n+4} & \rTo &\widetilde{Q}_{n+3}\\
\cdots  &\dTo_{-e_{j,2}} & &\dTo_{(-1)^{n+1} a_{2}e_{k-1,4}} & & \dTo_{- e_{j,2}}\\
   &\widetilde{Q}_3 & \rTo & \widetilde{Q}_2 & \rTo & \widetilde{Q}_1\\
\end{diagram}

For $2\le \ell\le n$, since $\lambda_2(\b_\ell\a_{\ell-1}\tensor \b_{\ell-1}\a_{\ell-2})=0$, we have by induction,
$$\l_{\ell+2}(\b_\ell\a_{\ell-1}\tensor\cdots\tensor\b_1\a_0\tensor\b_0\tensor\g_0\cdots\g_n)=$$
$$-\l_2\left(\b_\ell\a_{\ell-1}\tensor G\l_{\ell+1}(\b_{\ell-1}\a_{\ell-2}\tensor\cdots\tensor\b_1\a_0\tensor\b_0\tensor\g_0\cdots\g_n)\right)$$

where $G\l_{\ell+2}(\b_\ell\a_{\ell-1}\tensor\cdots\tensor\b_1\a_0 \tensor \b_0\tensor \g_0\cdots \g_n)=$
\begin{diagram}
&  \widetilde{Q}_{n+\ell+4} & \rTo & \widetilde{Q}_{n+\ell+3} & \rTo &\widetilde{Q}_{n+\ell+2}\\
\cdots &\dTo_{{X(\ell)} e_{x(\ell),\ell+1}} & &\dTo_{ {Y(\ell)} a_{\ell+1}e_{y(\ell),2(\ell+1)}} & & \dTo_{ {X(\ell)}e_{x(\ell),\ell+1}}\\
  &\widetilde{Q}_3 & \rTo & \widetilde{Q}_2 & \rTo & \widetilde{Q}_1\\
\end{diagram}

The matrix indices are $x(\ell)=\frac{1}{2}\left(k+j+(-1)^{\ell}(k-j)\right)-\left\lfloor\dfrac{\ell}{2}\right\rfloor$ and $y(\ell)=x(\ell+1)$, and the signs, which repeat with period 4, are determined by $Y(\ell)=(-1) \widehat{\quad}{\left\lfloor \dfrac{2\ell+1+(-1)^n}{4}\right\rfloor}$ and $X(\ell)=Y(\ell+1)$. 

\begin{thm} 
The algebra $E(B)$ admits a canonical $A_{\infty}$-structure for which $m_{j+3}$ is nonzero for all $0\le j\le n$. 
\end{thm}

\begin{proof}
By Proposition \ref{Merkulov}, $m_{n+3}=p\lambda_{n+3}i$. We will show that
$$
m_{n+3}(\a_n\otimes \b_n\a_{n-1}\otimes\cdots\otimes \b_1\a_0\otimes \b_0\otimes \g_0\cdots \g_n)\neq 0.
$$

Since $pG=0$, we have
\begin{align*}
p\l_{n+3} &= p\l_2 \sum_{ \substack{s+t = n+3 \\
s, t \geq 1}}
(-1)^{s+1}(G\l_s \tensor G\l_t)\\
&=p\lambda_2(G\lambda_1\otimes G\lambda_{n+2})\\
&=-p\lambda_2(1\otimes G\lambda_{n+2}).
\end{align*}

By the definitions of $G$ above, we see that
$$-p\l_2(\a_n\tensor G\l_{n+2}(\b_n\a_{n-1}\tensor\cdots\tensor\b_1\a_0\tensor\b_0\tensor\g_0\cdots\g_n))$$ $$=(-1)^{X(n)+1}\m_{2n+2}^{x(n)}$$
where $\m_{2n+2}^{x(n)}$ is a graded dual basis element in $E^{2n+2}(A)$. Hence 
$$m_{n+3}(\a_n\otimes \b_n\a_{n-1}\otimes\cdots\otimes\b_1\a_0\otimes \b_0\otimes \g_0\cdots \g_n)\neq 0.$$

By choosing a cocycle $i(\a_j)$, it is straightforward to check that with the above definitions,  
$m_{j+3}(\a_j\otimes \b_j\a_{j-1}\otimes\cdots\otimes \b_1\a_0\otimes \b_0\otimes \g_0\cdots \g_n)\neq 0$ for all $0\le j \le n$. 

\end{proof}

\section{Detecting the $\K_2$ Condition}

In this section we present two monomial $K$-algebras whose Yoneda algebras admit very similar canonical $A_{\infty}$-structures. Only one of the algebras is $\K_2$. These examples illustrate that the $\K_2$ property is not detected by any obvious vanishing patterns among higher multiplications. 

Let $V$ be the $K$-vector space spanned by $\{x, y, z, w\}$. Let $W^1\subset T(V)$ be the $K$-linear subspace on basis  $\{y^2zx, zx^2, y^2w\}$. Let $W^2\subset T(V)$ be the $K$-linear subspace on basis $\{ y^2z, zx^2, y^2w^2\}$ . Let $A^1=T(V)/\left<W^1\right>$ and $A^2=T(V)/\left<W^2\right>$.
The algebra $A^2$ was studied in \cite{CS}.

In Section 5 of \cite{CS}, the authors give an algorithm for producing a minimal projective resolution of the trivial module for a monomial $K$-algebra. Applying the algorithm, we obtain minimal projective resolutions of $_{A^1}K$ and $_{A^2}K$ of the form
$$ 0\rightarrow A^t(-5)\xrightarrow{M^t_3} A^t(-3,-3,-4)\xrightarrow{M^t_2}  A^t(-1)^{\oplus 4}\xrightarrow{M^t_1} A^t\rightarrow K\rightarrow 0$$
for $t\in\{1,2\}$ where
$$M^1_3=M^2_3=\begin{pmatrix} 0 & y^2 & 0\\ \end{pmatrix}$$
$$
M^1_2=\begin{pmatrix} 0 & 0 & 0 & y^2\\ zx & 0 & 0 & 0\\ y^2z & 0 & 0 & 0\\ \end{pmatrix}\quad
M^2_2=\begin{pmatrix} 0 & 0 & y^2 & 0\\ zx & 0 & 0 & 0\\ 0 & 0 & 0 & y^2w\\ \end{pmatrix}
$$
$$M^1_1=M^2_1=\begin{pmatrix} x & y & z & w\\ \end{pmatrix}^T.$$

Theorem \ref{K2criterion} shows that $A^1$ is a $\K_2$-algebra and $A^2$ is not, because $y^2zx$ is an essential relation in $A^1$ but not in $A^2$. We note that this implies all Yoneda products on $E(A^2)$ are zero. 

We make the identifications $E^1(A^1)=V^*=E^1(A^2)$, $E^2(A^1)=(W^1)^*$, and $E^2(A^2)=(W^2)^*$.
We choose $K$-bases for $E(A^1)$ and $E(A^2)$ as follows. Let $X$, $Y$, $Z$, and $W$ denote the $K$-linear duals of $x$, $y$, $z$, and $w$ respectively. Let $R^1_1$, $R^1_2$, and $R^1_3$ be the $K$-linear duals of the tensors $y^2w$, $zx^2$, $y^2zx\in W^1$, respectively. Let $R^2_1$, $R^2_2$, and $R^2_3$ be the $K$-linear duals of the tensors $y^2z$, $zx^2$, $y^2w^2\in W^2$, respectively.  For $t\in\{1,2\}$, let $\alpha^t$ be a nonzero vector in $E^{3,5}(B^t)$. We denote the $A_{\infty}$-structure map on $E(A^t)^{\otimes i}$ by $m^t_i$.

A standard calculation in the algebra $E(A^1)$ determines the Yoneda product ${m^1_2(R^1_3\otimes X)=c\alpha^1}$ for some $c\in K^*$. It is straightforward to check that Yoneda products of all other pairs of our specified basis vectors are zero. 

We note that, for degree reasons, the only higher multiplications on $E=E(A^t)$ which could be nonzero are the following.
\begin{align*}
m^t_4&: (E^1)^{\otimes 4}\rightarrow E^{2,4} &
m^t_3&: (E^1)^{\otimes 3}\rightarrow E^{2,3}\\
m^t_3&: E^1\otimes E^1\otimes E^{2,3}\rightarrow E^{3,5}&
m^t_3&: E^1\otimes E^{2,3}\otimes E^1\rightarrow E^{3,5}\\
m^t_3&: E^{2,3}\otimes E^1\otimes E^1\rightarrow E^{3,5}&
\end{align*}

We recall from the introduction that the $A_{\infty}$-structure $\{m_i\}$ on a graded $K$-vector space $E$ determines a $K$-linear map $\oplus m_i: T(E)_+\rightarrow E$.

\begin{prop}
\label{MonomialExample}
There exists a canonical $A_{\infty}$-structure on $E(A^1)$ such that the map $\oplus m^1_i$ vanishes on all monomials of $T(E(A^1))_+$ except $$\{YYZX,\ ZXX,\ YYW,\ YYR^1_2,\ R^1_3X\}.$$
There exists a canonical $A_{\infty}$-structure on $E(A^2)$ such that the map $\oplus m^2_i$ vanishes on all monomials of $T(E(A^2))_+$ except
$$\{YYWW,\ ZXX,\ YYZ,\ YYR^2_2,\ R^2_1XX\}.$$
\end{prop}

We remark that these canonical structures are not the only canonical $A_{\infty}$-structures on $E(A^1)$ and $E(A^2)$. However, it can be shown that the map $\oplus m^t_i$ associated to any other canonical structure on $E(A^t)$ is non-vanishing on the vectors listed in the proposition.

It is well known (see \cite{LPWZ2}, \cite{Keller1}) that there is a canonical $A_{\infty}$-structure on the Yoneda algebra of a graded algebra $A=T(V)/I$ such that the restriction of $m_n$ to $E^1(A)^{\otimes n}$ is dual to the natural inclusion of degree $n$ essential relations $(I/I')_n\hookrightarrow V^{\otimes n}$. If we choose such an $A_{\infty}$-structure for $E(A^1)$, the Stasheff identity $\mathbf{SI}(5)$ determines the remaining structure on $E(A^1)$, giving the first part of the proposition. We prove the proposition for $E(A^2)$.

\begin{proof}
Let $Q_{\bullet}$ be the minimal projective resolution of $_{A^2}K$ given above, with $Q_0=A^2$. Let $e_{i,j}$ denote the $(i,j)$ matrix unit.  We define maps $i$ and $G$ as in Section 4. When defining $G$, we suppress the inclusion maps $i$. We recall that $\lambda_2$ is multiplication in $\End_{A^2}(Q_{\bullet})$ and $\lambda_3=\lambda_2(G\lambda_2\otimes 1-1\otimes G\lambda_2)$. We define
$$i(X)=
\begin{diagram}[width=20pt]
Q_2 & \rTo &  Q_1\\
\dTo_{-ze_{2,1}}  & & \dTo_{e_{1,1}}\\
Q_1 & \rTo & Q_0\\
\end{diagram}
\quad i(Y)=
\begin{diagram}
Q_1\\
\dTo_{e_{2,1}}\\
Q_0\\
\end{diagram}
\quad i(Z)=
\begin{diagram}[width=20pt]
 Q_2 & \rTo &  Q_1\\
 \dTo_{-ye_{1,2}}  & & \dTo_{e_{3,1}}\\
 Q_1 & \rTo & Q_0\\
\end{diagram}
$$

$$
i(W)=
\begin{diagram}[width=20pt]
 Q_2 & \rTo &  Q_1\\
 \dTo_{-y^2e_{3,4}}  & & \dTo_{e_{4,1}}\\
 Q_1 & \rTo & Q_0\\
\end{diagram}
\quad
 i(R^2_1)=
\begin{diagram}
Q_2\\
\dTo_{e_{1,1}}\\
Q_0\\
\end{diagram}
\quad i(R^2_2)=
\begin{diagram}[width=20pt]
 Q_3 & \rTo &  Q_2\\
 \dTo_{ye_{1,2}}  & & \dTo_{e_{2,1}}\\
 Q_1 & \rTo & Q_0\\
\end{diagram}
$$

$$
  i(R^2_3)=
\begin{diagram}
Q_2\\
\dTo_{e_{3,1}}\\
Q_0\\
\end{diagram}
\quad G\lambda_2(XX)=
\begin{diagram}[width=20pt]
 Q_3 & \rTo &  Q_2\\
 \dTo_{e_{1,1}}  & & \dTo_{-e_{2,3}}\\
 Q_2 & \rTo & Q_1\\
\end{diagram}
\quad G\lambda_2(YZ)=
\begin{diagram}
Q_2\\
\dTo_{-e_{1,2}}\\
Q_1\\
\end{diagram}
$$

$$G\lambda_2(YR^2_2)=
\begin{diagram}
Q_3\\
\dTo_{e_{1,2}}\\
Q_1\\
\end{diagram}
\quad G\lambda_2(WW)=
\begin{diagram}
Q_2\\
\dTo_{-ye_{3,2}}\\
Q_1\\
\end{diagram}
\quad G\lambda_3(YWW)=
\begin{diagram}
Q_2\\
\dTo_{e_{3,2}}\\
Q_1\\
\end{diagram}
$$

We recall that $m_i=p\lambda_i$. With the definitions above, it is straightforward to check that $m_3(ZXX)$, $m_3(YYZ)$, $m_3(YYR^2_2)$ and $m_3(R^2_1XX)$ are nonzero. Recalling that 
$p\lambda_4(YYWW)=-p\lambda_2(YG\lambda_3(YWW))$, we see that $m_4(YYWW)$ is nonzero.

It remains to check that all other higher multiplications are zero. This is easily done by inspection.

\end{proof}

\bibliographystyle{amsplain}
\bibliography{bibliog}

\end{document}